\documentclass[12pt,oneside,reqno]{amsart}

\usepackage{amsmath,amssymb,amsthm,textcomp}
\usepackage{amsfonts,graphicx}
\usepackage[mathscr]{eucal}
\usepackage{color}
\usepackage{csquotes}
\usepackage[backend=bibtex,%
firstinits=true,%
doi=false,%
isbn=true,%
url=false,%
maxnames=99]{biblatex}%

\AtEveryBibitem{\clearfield{issn}}
\AtEveryCitekey{\clearfield{issn}}
\numberwithin{equation}{section}
\DeclareNameAlias{sortname}{last-first}
\theoremstyle{definition}

\numberwithin{equation}{section}

\newcommand{\ncom}{\newcommand}

\ncom{\beq}{\begin{equation}}
\ncom{\eeq}{\end{equation}}
\ncom{\bea}{\begin{eqnarray*}}
\ncom{\eea}{\end{eqnarray*}}
\ncom{\beqa}{\begin{eqnarray}}
\ncom{\eeqa}{\end{eqnarray}}
\ncom{\nno}{\nonumber}
\ncom{\non}{\nonumber}
\ncom{\ds}{\displaystyle}
\ncom{\half}{\frac{1}{2}}
\ncom{\mbx}{\makebox{.25cm}}
\ncom{\hs}{\mbox{\hspace{.25cm}}}
\ncom{\rar}{\rightarrow}
\ncom{\Rar}{\Rightarrow}
\ncom{\noin}{\noindent}
\ncom{\bc}{\begin{center}}
\ncom{\ec}{\end{center}}
\ncom{\sz}{\scriptsize}
\ncom{\rf}{\ref}
\ncom{\s}{\sqrt{2}}
\ncom{\sgm}{\sigma}
\ncom{\Sgm}{\Sigma}
\ncom{\psgm}{\sigma^{\prime}}
\ncom{\dt}{\delta}
\ncom{\Dt}{\Delta}
\ncom{\lmd}{\lambda}
\ncom{\Lmd}{\Lambda}
\ncom{\Th}{\Theta}
\ncom{\e}{\eta}
\ncom{\eps}{\epsilon}
\ncom{\pcc}{\stackrel{P}{>}}
\ncom{\lp}{\stackrel{L_{p}}{>}}
\ncom{\dist}{{\rm\,dist}}
\ncom{\sspan}{{\rm\,span}}
\ncom{\re}{{\rm Re\,}}
\ncom{\im}{{\rm Im\,}}
\ncom{\sgn}{{\rm sgn\,}}
\ncom{\ba}{\begin{array}}
\ncom{\ea}{\end{array}}
\ncom{\hone}{\mbox{\hspace{1em}}}
\ncom{\htwo}{\mbox{\hspace{2em}}}
\ncom{\hthree}{\mbox{\hspace{3em}}}
\ncom{\hfour}{\mbox{\hspace{4em}}}
\ncom{\vone}{\vskip 2ex}
\ncom{\vtwo}{\vskip 4ex}
\ncom{\vonee}{\vskip 1.5ex}
\ncom{\vthree}{\vskip 6ex}
\ncom{\vfour}{\vspace*{8ex}}
\ncom{\norm}{\|\;\;\|}
\ncom{\integ}[4]{\int_{#1}^{#2}\,{#3}\,d{#4}}
\ncom{\vspan}[1]{{{\rm\,span}\{ #1 \}}}
\ncom{\dm}[1]{ {\displaystyle{#1} } }
\ncom{\ri}[1]{{#1} \index{#1}}

\newtheorem{theorem}{\bf Theorem}[section]
\newtheorem{remark}{\bf Remark}[section]

\newtheorem{proposition}{Proposition}[section]

\newtheorem{corollary}{Corollary}[section]

\newtheoremstyle
    {remarkstyle}
    {}
    {11pt}
    {}
    {}
    {\bfseries}
    {:}
    {     }
    {\thmname{#1} \thmnumber{#2} }

\theoremstyle{remarkstyle}

\def\eps{\varepsilon}

\begin{document}
\title{On Densities of the Product, Quotient and Power of Independent Subordinators}
\author[Kuldeep Kumar Kataria]{K. K. Kataria}
\address{Kuldeep Kumar Kataria, Department of Mathematics,
 Indian Institute of Technology Bombay, Powai, Mumbai 400076, INDIA.}
 \email{kulkat@math.iitb.ac.in}
\author{P. Vellaisamy}
\address{P. Vellaisamy, Department of Mathematics,
 Indian Institute of Technology Bombay, Powai, Mumbai 400076, INDIA.}
 \email{pv@math.iitb.ac.in}
\thanks{The research of K. K. Kataria was supported by a UGC fellowship no. F.2-2/98(SA-I), Govt. of India.}
\subjclass[2010]{Primary : 60E07; Secondary : 33C60}
\keywords{Stable subordinator; tempered stable subordinator; exist-times; $H$-function.}
\begin{abstract}
We obtain the closed form expressions for the densities of the product, quotient, power and scalar multiple of independent stable subordinators. Similar results for the independent inverse stable subordinators are discussed. This is achieved by expressing the densities of stable and inverse stable subordinators in terms of the Fox's $H$-function.
\end{abstract}

\maketitle
\section{Introduction}
\setcounter{equation}{0}
A stable subordinator $D_\beta(t)$, with index $\beta$, is a one  dimensional stable L{\'e}vy process with non-decreasing sample paths (see \cite[pp. 49-52]{Applebaum2004}). Subordinated processes have interesting probabilistic properties with applications in financial time series. The inverse stable subordinator is defined by $E_\beta(t)=\inf\{x>0:D_\beta(x)>t\}$. As statistical densities are basically elementary functions or the products of such functions, the theory of special functions is directly related to statistical distribution theory. The $H$-function, introduced by Fox \cite{Fox395}, has numerous applications in the theory of statistical distributions and physical science. Some areas of astrophysics where the $H$-function appears naturally are the analytic solar model, pathway analysis, gravitational instability and reaction-diffusion problems (for more details on the $H$-function and its applications see \cite{Mathai2010}).

The main aim of this paper is to show that the random variable defined as the product, quotient, power or scalar multiple of independent stable subordinators has $H$-distribution. For this purpose we first obtain the closed form expression for the density of a $\beta$-stable subordinator in terms of the Fox's $H$-function for $\beta\in(0,1)$. Some known results such as the self-similarity property of a stable subordinator \textit{etc.} are evident from such closed form. Similar representations for the densities of the tempered stable subordinator and the first-exit time of a stable subordinator, also known as the inverse stable subordinator, are obtained. Our findings complement and extend the results of Meerschaert and Scheffler \cite{Mark2004}, Meerschaert and Straka  \cite{Mark2013}.

\section{Preliminaries}
\setcounter{equation}{0}
We start with some definitions and identities, and set some notations required in the paper.

\noindent \textbf{The Fox's $H$-function.} This function is represented by the following Mellin-Barnes type contour integral (see \cite[p. 2]{Mathai2010})
\begin{equation}\label{1.2}
H(z)=\mathrm{H}^{m,n}_{p,q}\Bigg[z\left|
\begin{matrix}
    (a_i,A_i)_{1,p}\\ 
    (b_j,B_j)_{1,q}
  \end{matrix}
\right.\Bigg]=\frac{1}{2\pi i}\int_{c-i\infty}^{c+i\infty}\chi(s)z^{-s}\,ds,
\end{equation}
where
\begin{equation}\label{1.3}
\chi(s)=\frac{\prod_{i=1}^{n}\Gamma\left(1-a_i-A_is\right)\prod_{j=1}^{m}\Gamma\left(b_j+B_js\right)}{\prod_{i=n+1}^{p}\Gamma\left(a_i+A_is\right)\prod_{j=m+1}^{q}\Gamma\left(1-b_j-B_js\right)}.
\end{equation}
\noindent An equivalent definition can be obtained on substituting $w=-s$ in (\ref{1.2}). In the above definition, $i=\sqrt{-1}$, $z\neq0$, and $z^{-s}=\exp\{-s(\ln|z|+i\arg z)\}$, where $\ln|z|$ represents the natural logarithm of $|z|$ and $\arg z$ is not necessarily the principal value. Also, $m, n, p, q$ are integers satisfying $0 \leq m \leq q$ and $0 \leq n \leq p$ with $A_i,B_j>0$ for $i=1,2,\ldots,p$, $j=1,2,\ldots,q$, and $a_i$'s and $b_j$'s are complex numbers. The path of integration, in the complex $s$-plane, runs from $c-i\infty$ to $c+i\infty$ for some real number $c$ such that the singularity of $\Gamma\left(b_j+B_js\right)$ for $j=1,2,\ldots,m$ lie entirely to the left of the path and the singularity of $\Gamma\left(1-a_i-A_is\right)$ for $i=1,2,\ldots,n$ lie entirely to the right of the path. An empty product is to be interpreted as unity.

\noindent \textbf{The $H$-function distribution.} This distribution was introduced by Carter and Springer \cite{Springer542} whose density is given by
\begin{equation}
f(x)=kH(\delta x)=\frac{\delta}{\chi(1)}\mathrm{H}^{m,n}_{p,q}\Bigg[\delta x\left|
\begin{matrix}
    (a_i,A_i)_{1,p}\\ 
    (b_j,B_j)_{1,q}
  \end{matrix}
\right.\Bigg],\ \ \ \ x>0,
\end{equation}
where $\delta\neq0$ and $k=\delta/\chi(1)$ is the normalizing constant. The cumulative distribution function of the $H$-function distribution is (see \cite[p. 110]{Bodenschatz1992}, \cite[p. 104]{Cook1981})
\begin{equation}\label{2.11p}
F(x)=\frac{1}{\chi(1)}\mathrm{H}^{m,n+1}_{p+1,q+1}\Bigg[\delta x\left|
\begin{matrix}
    (1,1)&(a_i+A_i,A_i)_{1,p}\\
    (b_j+B_j,B_j)_{1,q}&(0,1)
      \end{matrix}
\right.\Bigg],
\end{equation}
provided $-B_j^{-1}b_j<1$ for all $j\in\{1,2,\ldots,m\}$. The Mellin and Laplace transforms of $H(\delta x)$ are given by (see \cite{Springer542}, \cite[pp. 47-50]{Mathai2010})
\begin{equation}\label{2.1p}
\mathcal{M}_x(H(\delta x))=\frac{\chi(s)}{\delta^s}
\end{equation}
and
\begin{equation}\label{2.2p}
\mathcal{L}_x(H(\delta x))=\frac{1}{\delta}\mathrm{H}^{n+1,m}_{q,p+1}\Bigg[\frac{s}{\delta}\left|
\begin{matrix}
    (1-b_j-B_j,B_j)_{1,q}\\ 
    (0,1)&(1-a_i-A_i,A_i)_{1,p}
  \end{matrix}
\right.\Bigg],
\end{equation}
respectively.

\noindent \textbf{Some properties of the $H$-function.} The following properties follow from the definition of the $H$-function (see \cite{Springer542}, \cite[pp. 11-12]{Mathai2010}). Let $\rho,\lambda$ be any complex numbers and $\sigma>0$. Then
\begin{eqnarray}
\mathrm{H}^{m,n}_{p,q}\Bigg[\frac{1}{z}\left|
\begin{matrix}
    (a_i,A_i)_{1,p}\\ 
    (b_j,B_j)_{1,q}
  \end{matrix}
\right.\Bigg]&=&\mathrm{H}^{n,m}_{q,p}\Bigg[z\left|
\begin{matrix}
    (1-b_j,B_j)_{1,q}\\
    (1-a_i,A_i)_{1,p}    
  \end{matrix}
\right.\Bigg],\label{2.8}\\
\mathrm{H}^{m,n}_{p,q}\Bigg[z^\sigma\left|
\begin{matrix}
    (a_i,A_i)_{1,p}\\ 
    (b_j,B_j)_{1,q}
  \end{matrix}
\right.\Bigg]&=&\sigma^{-1}\mathrm{H}^{m,n}_{p,q}\Bigg[z\left|
\begin{matrix}
    \left(a_i,\sigma^{-1}A_i\right)_{1,p}\\ 
    \left(b_j,\sigma^{-1}B_j\right)_{1,q}
  \end{matrix}
\right.\Bigg],\label{2.9}\\
z^\rho \mathrm{H}^{m,n}_{p,q}\Bigg[z\left|
\begin{matrix}
    (a_i,A_i)_{1,p}\\ 
    (b_j,B_j)_{1,q}
  \end{matrix}
\right.\Bigg]&=&\mathrm{H}^{m,n}_{p,q}\Bigg[z\left|
\begin{matrix}
    (a_i+\rho A_i,A_i)_{1,p}\\ 
    (b_j+\rho B_j,B_j)_{1,q}
  \end{matrix}
\right.\Bigg].\label{2.10}
\end{eqnarray}
Also, the $k$-th derivative of the $H$-function is given by (see \cite[p. 13]{Mathai2010})
\begin{equation}\label{2.11}
\frac{\mathrm{d}^k}{\mathrm{d}z^k}z^{\rho-1}\mathrm{H}^{m,n}_{p,q}\Bigg[\lambda z^{\sigma}\left|
\begin{matrix}
    (a_i,A_i)_{1,p}\\ 
    (b_j,B_j)_{1,q}
  \end{matrix}
\right.\Bigg]
=(-1)^kz^{\rho-k-1}\mathrm{H}^{m+1,n}_{p+1,q+1}\Bigg[\lambda z^{\sigma}\left|
\begin{matrix}
    (a_i,A_i)_{1,p}&(1-\rho,\sigma)\\ 
    (1-\rho+k,\sigma)&(b_j,B_j)_{1,q}
  \end{matrix}
\right.\Bigg].
\end{equation}
The following identity will be used later (see \cite[p. 23]{Mathai2010})
\begin{equation}\label{2.7}
\mathrm{H}^{1,0}_{0,1}\Bigg[z\left|
\begin{matrix}
     \\ 
    (b,B)
  \end{matrix}
\right.\Bigg]=\frac{1}{B}z^{b/B}e^{-z^{1/B}}.
\end{equation}
\noindent \textbf{Generalized Wright function and $M$-Wright function.} Let $\alpha_i,\beta_j\in\mathbb{R}\setminus\{0\}$ and $a_i,b_
j\in\mathbb{C}$, $i=1,2,\ldots,p;\ j=1,2,\ldots,q$. Then the generalized Wright function (see \cite[Eq. (1.1)]{Kilbas2002437}) is defined by
\begin{equation*}\label{1.1ku}
{}_{p}\Psi_q(z)={}_{p}\Psi_q\Bigg[z\left|
\begin{matrix}
    (a_i,\alpha_i)_{1,p}\\ 
    (b_j,\beta_j)_{1,q}
  \end{matrix}
\right.\Bigg]
:=\sum_{n=0}^{\infty}\frac{\prod_{i=1}^{p}\Gamma(a_i+n\alpha_i)}{\prod_{j=1}^{q}\Gamma(b_j+n\beta_j)}\frac{z^n}{n!},\ \ z\in\mathbb{C}.
\end{equation*}
The integral representation of ${}_{p}\Psi_q(z)$ in terms of Mellin-Barnes integral (see \cite[Eq. (1.5)]{Kilbas2002437}) is given by
\begin{equation}\label{kpp}
{}_{p}\Psi_q\Bigg[z\left|
\begin{matrix}
    (a_i,\alpha_i)_{1,p}\\ 
    (b_j,\beta_j)_{1,q}
  \end{matrix}
\right.\Bigg]=\frac{1}{2\pi i}\int_{c-i\infty}^{c+i\infty}\frac{\Gamma(s)\prod_{i=1}^{p}\Gamma(a_i-\alpha_i s)}{\prod_{j=1}^{q}\Gamma(b_j-\beta_j s)}(-z)^{-s}\,ds.
\end{equation}
The $M$-Wright function $M_\beta(z)$, $0<\beta<1$, (see \cite[Appendix A.5]{Gorenflo2015}) is defined as
\begin{equation*}\label{1.1uu}
M_\beta(z)={}_{0}\Psi_1\Bigg[-z\left|
\begin{matrix}
    \\ 
    (1-\beta,-\beta)
  \end{matrix}
\right.\Bigg]=\sum_{n=0}^{\infty}\frac{1}{\Gamma(1-\beta-n\beta)}\frac{(-z)^n}{n!},\ \ z\in\mathbb{C}.
\end{equation*}
The Euler's reflection and the Legendre duplication formula for the gamma function are
\begin{equation}\label{2.7sd}
\Gamma(z)\Gamma(1-z)=\frac{\pi}{\sin z\pi}
\end{equation}
and
\begin{equation}\label{2.7na}
\Gamma(2z)=\frac{2^{2z-1}}{\sqrt{\pi}}\Gamma(z)\Gamma(z+1/2),
\end{equation}
respectively.
\section{Independent stable subordinators}
\setcounter{equation}{0}
From L\'evy-Khinchin representation (see \cite[pp. 49-50]{Applebaum2004}), the Laplace transform of a subordinator $D(t)$ with density $f(x,t)$ is given by the following expression:
\begin{equation*}
\mathcal{L}_x(f(x,t))=\mathbb{E}\left(e^{-sD(t)}\right)=\int_0^\infty f(x,t)e^{-sx}\,\mathrm{d}x=e^{-t\psi(s)},
\end{equation*}
where $\psi(s)$ denotes the Laplace exponent given by
\begin{equation*}
\psi(s)=bs+\int_0^\infty \left(1-e^{-sx}\right)\,\lambda(\mathrm{d}x),\ \ \ \ s>0.
\end{equation*}
Here $b\geq 0$ denotes the drift whereas $\lambda$ is a L\'evy measure. For the following L\'evy measure
\begin{equation*}
{\lambda_\beta(\mathrm{d}x)}=\frac{\beta}{\Gamma(1-\beta)x^{1+\beta}}\mathbb{I}_{(0,\infty)}(x)\,{\mathrm{d}x},\ \ \ \ 0<\beta<1,
\end{equation*}
the identity $s^\beta=\int_0^\infty \left(1-e^{-sx}\right)\,\lambda_\beta (\mathrm{d}x)$ holds. Since a $\beta$-stable subordinator $D_\beta(t)$, $0<\beta<1$, is a subordinator with zero drift and L\'evy measure $\lambda_\beta$, its Laplace exponent is $\psi(s)=s^\beta$. This leads to the integral representation for the density $f_\beta (x,t)$ of a $\beta$-stable subordinator.
\begin{proposition}\label{t1}
The density of a $\beta$-stable subordinator, $0<\beta<1$, admits the following integral form
\begin{equation*}
f_\beta (x,t)=\frac{1}{\pi}\int_0^{\infty}e^{-ux-tu^\beta\cos\beta\pi}\sin(tu^\beta\sin\beta\pi)\,\mathrm{d}u,\ \ \ \ x>0.
\end{equation*}
\end{proposition}
\begin{proof}
The proof follows on similar lines as the proof of Theorem 2.1 in \cite{Arun2015} by choosing a contour similar to Figure 1 with branch point at origin (see Remark 2.1 in \cite{Arun2015}).
\end{proof}
\noindent Next we obtain the Mellin transform of the density of a stable subordinator.
\begin{proposition}\label{t2}
The Mellin transform of the density of a $\beta$-stable subordinator, $0<\beta<1/2$, is
\begin{equation*}
\mathcal{M}_x\left(f_\beta (x,t)\right)=\frac{1}{\beta t^{(1-s)/\beta}}\frac{\Gamma\left((1-s)/\beta\right)}{\Gamma\left(1-s\right)}.
\end{equation*}
\end{proposition}
\begin{proof}
See Appendix \textbf{A.1}.
\end{proof}
\begin{remark}
The integral representation and the Mellin transform for the density of a stable distribution is given in \cite[Theorem 2.2.1, p. 70]{Zolotarev1986} and \cite[Theorem 2.6.3, p. 118]{Zolotarev1986}, respectively. The Fox function representation for the density of a stable distribution is given by Schneider \cite{Schneider1986}.
\end{remark}
Now by using the Mellin inversion formula and Proposition \ref{t2}, the closed form representation of the density of a $\beta$-stable subordinator in terms of the Fox's $H$-function is given by
\begin{eqnarray}\label{zzz}
f_\beta (x,t)&=&\frac{1}{2\pi i}\int_{c-i\infty}^{c+i\infty}\mathcal{M}_x\left(f_\beta (x,t)\right)x^{-s}\,\mathrm{d}s\nonumber\\
&=&\frac{1}{\beta t^{1/\beta}}\frac{1}{2\pi i}\int_{c-i\infty}^{c+i\infty}\frac{\Gamma\left(1/\beta-s/\beta\right)}{\Gamma\left(1-s\right)}\left(\frac{x}{t^{1/\beta}}\right)^{-s}\,\mathrm{d}s\nonumber\\
&=&\frac{1}{\beta t^{1/\beta}}\mathrm{H}^{0,1}_{1,1}\Bigg[\frac{x}{t^{1/\beta}}\left|
\begin{matrix}
    \left(1-1/\beta,1/\beta\right)\\ 
    (0,1)
  \end{matrix}
\right.\Bigg],\ \ x>0,\ \ 0<\beta<1/2.
\end{eqnarray}
The density of the L{\'e}vy subordinator ($\beta=1/2$) is given by the L{\'e}vy distribution (see \cite[p. 50]{Applebaum2004}) as follows:
\begin{eqnarray*}
f_{\frac{1}{2}}(x,t)&=&\frac{t}{2\sqrt{\pi}}x^{-3/2}e^{-t^2/4x},\ \ x>0\\
&=&\frac{4}{t^2\sqrt{\pi}}\mathrm{H}^{1,0}_{0,1}\Bigg[\frac{t^{2}}{4x}\left|
\begin{matrix}
    \\
    \left(3/2,1\right)
\end{matrix}
\right.\Bigg],\ \ \ \ \mathrm{(using\ (\ref{2.7}))}\nonumber\\
&=&\frac{4}{t^2\sqrt{\pi}}\mathrm{H}^{0,1}_{1,0}\Bigg[\frac{4x}{t^{2}}\left|
\begin{matrix}
    \left(-1/2,1\right)\\
    {}
\end{matrix}
\right.\Bigg],\ \ \ \ \mathrm{(using\ (\ref{2.8}))}\nonumber\\
&=&\frac{4}{ t^{2}\sqrt{\pi}}\frac{1}{2\pi i}\int_{c-i\infty}^{c+i\infty}\Gamma\left(3/2-s\right)\left(\frac{4x}{t^{2}}\right)^{-s}\,\mathrm{d}s\\
&=&\frac{2}{ t^{2}}\frac{1}{2\pi i}\int_{c-i\infty}^{c+i\infty}\frac{\Gamma\left(2(1-s)\right)}{\Gamma\left(1-s\right)}\left(\frac{x}{t^{2}}\right)^{-s}\,\mathrm{d}s,\ \ \ \ \mathrm{(using\ (\ref{2.7na}))}\nonumber\\
&=&\frac{2}{t^{2}}\mathrm{H}^{0,1}_{1,1}\Bigg[\frac{x}{t^{2}}\left|
\begin{matrix}
    \left(-1,2\right)\\ 
    (0,1)
  \end{matrix}
\right.\Bigg],
\end{eqnarray*}
\textit{i.e.} (\ref{zzz}) can be extended for index $\beta\in(0,1/2]$.

Next we see that this result can be extended for index $0<\beta<1$. Gorenflo and Mainardi \cite[Eq. (4.7)]{Gorenflo2015} express the density of $D_\beta(t)$, $0<\beta<1$, in terms of the $M$-Wright function. Hence,
\begin{eqnarray*}
f_\beta (x,t)&=&\beta tx^{-(\beta+1)}M_\beta(tx^{-\beta}),\ \ x>0\\
&=&\beta tx^{-(\beta+1)}{}_{0}\Psi_1\Bigg[-tx^{-\beta}\left|
\begin{matrix}
    \\ 
    (1-\beta,-\beta)
  \end{matrix}
\right.\Bigg]\\
&=&\beta tx^{-(\beta+1)}\frac{1}{2\pi i}\int_{c-i\infty}^{c+i\infty}\frac{\Gamma\left(s\right)}{\Gamma\left(1-\beta+\beta s\right)}\left(tx^{-\beta}\right)^{-s}\,\mathrm{d}s,\ \ \ \ \mathrm{(using\ (\ref{kpp}))}\\
&=&\beta tx^{-(\beta+1)}\mathrm{H}^{1,0}_{1,1}\Bigg[tx^{-\beta}\left|
\begin{matrix}
    \left(1-\beta,\beta\right)\\ 
    (0,1)
  \end{matrix}
\right.\Bigg]\\
&=&\beta tx^{-(\beta+1)}\mathrm{H}^{0,1}_{1,1}\Bigg[\frac{x^{\beta}}{t}\left|
\begin{matrix}
    \left(1,1\right)\\ 
    (\beta,\beta)
  \end{matrix}
\right.\Bigg],\ \ \ \ \mathrm{(using\ (\ref{2.8}))}\\
&=&\frac{\beta}{x}\mathrm{H}^{0,1}_{1,1}\Bigg[\frac{x^{\beta}}{t}\left|
\begin{matrix}
    \left(0,1\right)\\ 
    (0,\beta)
  \end{matrix}
\right.\Bigg],\ \ \ \ \mathrm{(using\ (\ref{2.10}))}\\
&=&\frac{1}{t^{1/\beta}}\mathrm{H}^{0,1}_{1,1}\Bigg[\frac{x}{t^{1/\beta}}\left|
\begin{matrix}
    \left(-1/\beta,1/\beta\right)\\ 
    (-1,1)
  \end{matrix}
\right.\Bigg],\ \ \ \ \mathrm{(using\ (\ref{2.9})\ and\ (\ref{2.10}))}\\
&=&\frac{1}{t^{1/\beta}}\frac{1}{2\pi i}\int_{c-i\infty}^{c+i\infty}\frac{\Gamma\left(1+1/\beta-1/\beta s\right)}{\Gamma\left(2-s\right)}\left(xt^{-1/\beta}\right)^{-s}\,\mathrm{d}s\\
&=&\frac{1}{\beta t^{1/\beta}}\frac{1}{2\pi i}\int_{c-i\infty}^{c+i\infty}\frac{\Gamma\left(1/\beta-1/\beta s\right)}{\Gamma\left(1-s\right)}\left(xt^{-1/\beta}\right)^{-s}\,\mathrm{d}s\\
&=&\frac{1}{\beta t^{1/\beta}}\mathrm{H}^{0,1}_{1,1}\Bigg[\frac{x}{t^{1/\beta}}\left|
\begin{matrix}
    \left(1-1/\beta,1/\beta\right)\\ 
    (0,1)
  \end{matrix}
\right.\Bigg]
\end{eqnarray*}
and thus (\ref{zzz}) can further be extended for index $\beta\in(0,1)$.

\noindent The above discussed results can be summarized in the following theorem. 
\begin{theorem}\label{t3}
The density of a $\beta$-stable subordinator $D_\beta(t)$, $0<\beta<1$, in terms of the $H$-function is given by
\begin{equation}\label{3.1aa}
f_\beta (x,t)=\frac{1}{\beta t^{1/\beta}}\mathrm{H}^{0,1}_{1,1}\Bigg[\frac{x}{t^{1/\beta}}\left|
\begin{matrix}
    \left(1-1/\beta,1/\beta\right)\\ 
    (0,1)
  \end{matrix}
\right.\Bigg],\ \ \ \ x>0.
\end{equation}
\end{theorem}
\begin{corollary}
The density $f_\beta (x,1)$ of $D_\beta(1)$, $0<\beta<1$, is
\begin{equation}\label{3.1aaa}
f_\beta (x,1)=\frac{1}{\beta}\mathrm{H}^{0,1}_{1,1}\Bigg[x\left|
\begin{matrix}
    \left(1-1/\beta,1/\beta\right)\\ 
    (0,1)
  \end{matrix}
\right.\Bigg],\ \ \ \ x>0.
\end{equation}
\end{corollary}
\begin{remark}\label{t6}
On using (\ref{2.2p}), (\ref{2.7}) and (\ref{3.1aa}), we can obtain a well known result \textit{i.e.} the Laplace transform of the density of a $\beta$-stable subordinator
\begin{equation*}
\mathcal{L}_x\left(f_\beta (x,t)\right)=\frac{1}{\beta}\mathrm{H}^{1,0}_{0,1}\Bigg[t^{1/\beta} s\left|
\begin{matrix}
    \\ 
    (0,1/\beta)
  \end{matrix}
\right.\Bigg]=e^{-ts^\beta}.
\end{equation*}
\end{remark}
\begin{proposition}\label{t4}
The Mellin and Laplace transforms of the density of a $\beta$-stable subordinator, $0<\beta<1$, with respect to the time variable are
\begin{equation*}
\mathcal{M}_t\left(f_\beta (x,t)\right)=\frac{1}{x^{1-\beta s}}\frac{\Gamma(s)}{\Gamma(\beta s)}
\end{equation*}
and 
\begin{equation*}
\mathcal{L}_t\left(f_\beta (x,t)\right)=x^{\beta-1}\mathrm{H}^{1,1}_{1,2}\Bigg[x^\beta s\left|
\begin{matrix}
    (0,1)\\ 
    (0,1)&(1-\beta,\beta)
  \end{matrix}
\right.\Bigg],
\end{equation*}
respectively.
\end{proposition}
\begin{proof} From (\ref{3.1aa}), we have
\begin{eqnarray}\label{3.1a}
f_\beta (x,t)&=&\frac{1}{\beta t^{1/\beta}}\mathrm{H}^{0,1}_{1,1}\Bigg[\left(\frac{x^\beta}{t}\right)^{1/\beta}\left|
\begin{matrix}
    \left(1-1/\beta,1/\beta\right)\\ 
    (0,1)
  \end{matrix}
\right.\Bigg]\nonumber\\
&=&\frac{1}{x}\left(\frac{x^\beta}{t}\right)^{1/\beta}\mathrm{H}^{0,1}_{1,1}\Bigg[\frac{x^\beta}{t}\left|
\begin{matrix}
    \left(1-1/\beta,1\right)\\ 
    (0,\beta)
  \end{matrix}
\right.\Bigg],\ \ \ \ \mathrm{(using\ (\ref{2.9}))}\nonumber\\
&=&\frac{1}{x}\mathrm{H}^{0,1}_{1,1}\Bigg[\frac{x^\beta}{t}\left|
\begin{matrix}
    \left(1,1\right)\\ 
    (0,\beta)
  \end{matrix}
\right.\Bigg],\ \ \ \ \mathrm{(using\ (\ref{2.10}))}\nonumber\\
&=&\frac{1}{x}\mathrm{H}^{1,0}_{1,1}\Bigg[\frac{t}{x^\beta}\left|
\begin{matrix}
    (0,\beta)\\ 
    (0,1)
  \end{matrix}
\right.\Bigg],\ \ \ \ \mathrm{(using\ (\ref{2.8}))}.
\end{eqnarray}
Finally by using (\ref{2.1p}) and (\ref{2.2p}), we get
\begin{equation*}
\mathcal{M}_t\left(f_\beta (x,t)\right)=\frac{1}{x}\mathcal{M}_t\left(\mathrm{H}^{1,0}_{1,1}\Bigg[\frac{t}{x^\beta}\left|
\begin{matrix}
    (0,\beta)\\ 
    (0,1)
  \end{matrix}
\right.\Bigg]\right)=\frac{1}{x^{1-\beta s}}\frac{\Gamma(s)}{\Gamma(\beta s)}
\end{equation*}
and
\begin{equation*}
\mathcal{L}_t\left(f_\beta (x,t)\right)=\frac{1}{x}\mathcal{L}_t\left(\mathrm{H}^{1,0}_{1,1}\Bigg[\frac{t}{x^\beta}\left|
\begin{matrix}
    (0,\beta)\\ 
    (0,1)
  \end{matrix}
\right.\Bigg]\right)=x^{\beta-1}\mathrm{H}^{1,1}_{1,2}\Bigg[x^\beta s\left|
\begin{matrix}
    (0,1)\\ 
    (0,1)&(1-\beta,\beta)
  \end{matrix}
\right.\Bigg].
\end{equation*}
\end{proof}
An alternate and direct proof of Proposition \ref{t4} for the Mellin transform is given in Appendix \textbf{A.2}. Note that on following this alternate proof for the Mellin transform and using the inverse Mellin formula, we can directly obtain the density of a stable subordinator in the form (\ref{3.1a}), which is equivalent to (\ref{3.1aa}).

In the next proposition, we express the cumulative distribution function $F_\beta (x,t)$ of a stable subordinator in terms of the Fox's $H$-function. This is a straightforward application of (\ref{2.11p}). 
\begin{proposition}\label{t7}
The cumulative distribution function of a $\beta$-stable subordinator, $0<\beta<1$, is
\begin{equation}\label{cdf1}
F_\beta (x,t)=\frac{1}{\beta}\mathrm{H}^{0,1}_{1,1}\Bigg[\frac{x}{t^{1/\beta}}\left|
\begin{matrix}
    (1,1/\beta)\\ 
    (0,1)
  \end{matrix}
\right.\Bigg].
\end{equation}
\end{proposition}
The next theorem gives the densities of the product, quotient, power and scalar multiple of independent stable subordinators. These are direct consequences of the results proved for the $H$-function distribution in Carter and Springer \cite{Springer542}.
\begin{theorem}
Let $D_{\beta_i}(t)$, $i=1,2,\ldots,n$, be $n$ independent $\beta_i$-stable subordinators, $0<\beta_i<1$, with corresponding density functions
\begin{equation*}
f_{\beta_i}(x_i,t)=\frac{1}{\beta_i t^{1/\beta_i}}\mathrm{H}^{0,1}_{1,1}\Bigg[\frac{x_i}{t^{1/\beta_i}}\left|
\begin{matrix}
    \left(1-1/\beta_i,1/\beta_i\right)\\ 
    (0,1)
  \end{matrix}
\right.\Bigg],\ \ \ \ x_i>0.
\end{equation*}
Then, the probability density function of\\
\noindent (i) $X(t)=aD_{\beta_i}(t),\ a>0,\ 1\leq i\leq n$, is given by
\begin{equation}\label{sym}
f(x,t)=\frac{1}{a{\beta_i} t^{1/{\beta_i}}}\mathrm{H}^{0,1}_{1,1}\Bigg[\frac{x}{at^{1/{\beta_i}}}\left|
\begin{matrix}
    \left(1-1/{\beta_i},1/{\beta_i}\right)\\ 
    (0,1)
  \end{matrix}
\right.\Bigg],\ \ x>0.
\end{equation}
\noindent (ii) $Y(t)=\prod_{i=1}^nD_{\beta_i}(t)$ is given by
\begin{equation*}
f(y,t)=\left(\prod_{i=1}^n\frac{1}{\beta_i t^{1/\beta_i}}\right)\mathrm{H}^{0,n}_{n,n}\Bigg[\frac{y}{t^{\sum_{i=1}^n1/\beta_i}}\left|
\begin{matrix}
    \left(1-1/\beta_i,1/\beta_i\right)_{1,n}\\ 
    (0,1)_{1,n}
  \end{matrix}
\right.\Bigg],\ \ y>0.
\end{equation*}
\noindent (iii) $W(t)=D_{\beta_i}^r(t),\ 1\leq i\leq n$, for $r$ rational is given by
\begin{equation*}
f(w,t)=\frac{1}{{\beta_i} t^{r/{\beta_i}}}\mathrm{H}^{0,1}_{1,1}\Bigg[\frac{w}{t^{r/{\beta_i}}}\left|
\begin{matrix}
    \left(1-r/{\beta_i},r/{\beta_i}\right)\\ 
    (1-r,r)
  \end{matrix}
\right.\Bigg],\ \  w>0,
\end{equation*}
when $r>0$ and for $r<0$
\begin{equation*}
f(w,t)=\frac{1}{{\beta_i} t^{r/{\beta_i}}}\mathrm{H}^{1,0}_{1,1}\Bigg[\frac{w}{t^{r/{\beta_i}}}\left|
\begin{matrix}
    (r,-r)\\
    (r/{\beta_i},-r/{\beta_i})
      \end{matrix}
\right.\Bigg],\ \ w>0.
\end{equation*}
\noindent (iv) $Z(t)=D_{\beta_i}(t)/D_{\beta_j}(t),\ 1\leq i,j\leq n$ and $i\neq j$, is given by
\begin{equation*}
f(z,t)=\frac{1}{\beta_i\beta_jt^{1/\beta_i+1/\beta_j}}\mathrm{H}^{1,1}_{2,2}\Bigg[\frac{z}{t^{1/\beta_i+1/\beta_j}}\left|
\begin{matrix}
    \left(1-1/\beta_i,1/\beta_i\right)&(-1,1)\\ 
    (-1/\beta_j,1/\beta_j)&(0,1)
  \end{matrix}
\right.\Bigg],\ \ z>0.
\end{equation*}
\end{theorem}
\begin{proof}
(i) The density of $X$ can be obtained on using the Mellin inversion formula as follows:
\begin{eqnarray*}
f(x,t)&=&\frac{1}{2\pi i}\int_{c-i\infty}^{c+i\infty}\mathcal{M}_{x}(f(x,t))x^{-s}\,ds\\
&=&\frac{1}{2\pi i}\int_{c-i\infty}^{c+i\infty}a^{s-1}\mathcal{M}_{x_i}(f_{\beta_i}(x_i,t))x^{-s}\,ds\\
&=&\frac{1}{a\beta_i t^{1/\beta_i}}\frac{1}{2\pi i}\int_{c-i\infty}^{c+i\infty}\frac{\Gamma\left(1/\beta_i-s/\beta_i\right)}{\Gamma\left(1-s\right)}\left(\frac{x}{at^{1/\beta_i}}\right)^{-s}\,\mathrm{d}s,\ \ \ \ x>0,
\end{eqnarray*}
where the last step follows from Proposition \ref{t2}. The proof now follows on using (\ref{1.2}). Parts (ii), (iii) and (iv) follow from Theorems 4.1, 4.2 and 4.3 of Carter and Springer \cite{Springer542}, respectively.
\end{proof}
\begin{remark}
Since $M$-Wright function is a particular case of $H$-function, (\ref{symm}) reduces to
\begin{equation}\label{symmzzx}
f(x,t)=a^{\beta_i}\beta_i tx^{-(\beta_i+1)}M_{\beta_i}\left(a^{\beta_i}tx^{-\beta_i}\right),\ \ x>0.
\end{equation}
On substituting $a=1/t^{1/\beta_i}$ in (\ref{sym}) or (\ref{symmzzx}), we can obtain the self similarity property of a stable subordinator \textit{i.e.} $D_{\beta_i}(t)/t^{1/\beta_i}{\overset{\mathcal{D}}{=}}D_{\beta_i}(1)$ for all $t>0$, where ${\overset{\mathcal{D}}{=}}$ means equal in distribution.
\end{remark}
\subsection{Tempered stable subordinator}
There is a limitation to the applications of a stable subordinator because of the non-existence of its finite moments. The tempered stable subordinator is obtained by exponentially tempering the distribution of a stable subordinator (see Rosinski \cite{Jan2007}). The density of the tempered stable subordinator $D_{\lambda,\beta}(t)$, $\lambda>0$, is given by
\begin{equation*}
f_{\lambda,\beta}(x,t)=e^{-\lambda x+\lambda^\beta t}f_{\beta}(x,t),\ \ \ \ x>0,
\end{equation*}
which on using (\ref{3.1aa}) can be expressed in terms of the $H$-function as follows:
\begin{equation}\label{kk32}
f_{\lambda,\beta}(x,t)=\frac{e^{-\lambda x+\lambda^\beta t}}{\beta t^{1/\beta}}\mathrm{H}^{0,1}_{1,1}\Bigg[\frac{x}{t^{1/\beta}}\left|
\begin{matrix}
    \left(1-1/\beta,1/\beta\right)\\ 
    (0,1)
  \end{matrix}
\right.\Bigg],\ \ \ \ x>0.
\end{equation}
The tempered stable subordinator have all moments finite and is infinitely divisible with exponentially decaying tail probability but lack self similarity. Next we give a different proof of a known result using the obtained density (\ref{kk32}).
\begin{proposition}\label{p2.22}
The Laplace transform of the density of the tempered $\beta$-stable subordinator, $0<\beta<1$, is
\begin{equation*}
\mathcal{L}_x\left(f_{\lambda,\beta}(x,t)\right)=e^{-t\left((\lambda+s)^\beta-\lambda^\beta\right)}.
\end{equation*}
\end{proposition}
\begin{proof}
From the definition of the Laplace transform, we have
\begin{eqnarray*}
\mathcal{L}_x\left(f_{\lambda,\beta}(x,t)\right)&=&\int_0^\infty f_{\lambda,\beta}(x,t)e^{-sx}\,\mathrm{d}x\\
&=&\frac{e^{\lambda^\beta t}}{\beta t^{1/\beta}}\int_0^\infty e^{-(\lambda+s)x}\mathrm{H}^{0,1}_{1,1}\Bigg[\frac{x}{t^{1/\beta}}\left|
\begin{matrix}
    \left(1-1/\beta,1/\beta\right)\\ 
    (0,1)
  \end{matrix}
\right.\Bigg]\,\mathrm{d}x\\
&=&\frac{e^{\lambda^\beta t}}{\beta}\mathrm{H}^{2,0}_{1,2}\Bigg[(\lambda+s)t^{1/\beta}\left|
\begin{matrix}
    (0,1)\\ 
    (0,1)&\left(0,1/\beta\right)
  \end{matrix}
\right.\Bigg],\ \ \ \ \mathrm{(using\ (\ref{2.2p}))}\\
&=&\frac{e^{\lambda^\beta t}}{\beta}\mathrm{H}^{1,0}_{0,1}\Bigg[(\lambda+s)t^{1/\beta}\left|
\begin{matrix}
    \\ 
    \left(0,1/\beta\right)
  \end{matrix}
\right.\Bigg],
\end{eqnarray*}
where the last equality is obtained on using the definition of the $H$-function. Finally, the result follows on using (\ref{2.7}).
\end{proof}
\noindent Similarly, the Laplace transform of the density of the tempered $\beta$-stable subordinator, $0<\beta<1$, with respect to the time variable is
\begin{equation*}
\mathcal{L}_t\left(f_{\lambda,\beta}(x,t)\right)=\frac{e^{-\lambda x}}{(s-\lambda^\beta)x}\mathrm{H}^{1,1}_{2,1}\Bigg[\frac{1}{x^\beta(s-\lambda^\beta)}\left|
\begin{matrix}
    (0,1)&\left(0,\beta\right)\\
    (0,1) 
    \end{matrix}
\right.\Bigg].
\end{equation*}
The proof follows on similar lines as the proof of the Proposition \ref{p2.22}.
\begin{proposition}
The Mellin transform of the density of the tempered $\beta$-stable subordinator, $0<\beta<1$, is
\begin{equation*}
\mathcal{M}_x\left(f_{\lambda,\beta}(x,t)\right)=\frac{e^{\lambda^\beta t}}{\beta t^{1/\beta}\lambda^s}\mathrm{H}^{0,2}_{2,1}\Bigg[\frac{1}{t^{1/\beta}\lambda}\left|
\begin{matrix}
    (1-s,1)&\left(1-1/\beta,1/\beta\right)\\
    (0,1) 
    \end{matrix}
\right.\Bigg].
\end{equation*}
\end{proposition}
\begin{proof}
From the definition of the Mellin transform, we have
\begin{eqnarray*}
\mathcal{M}_x\left(f_{\lambda,\beta}(x,t)\right)&=&\int_0^\infty x^{s-1}f_{\lambda,\beta}(x,t)\,\mathrm{d}x\\
&=&\frac{e^{\lambda^\beta t}}{\beta t^{1/\beta}}\int_0^\infty x^{s-1}e^{-\lambda x}\mathrm{H}^{0,1}_{1,1}\Bigg[\frac{x}{t^{1/\beta}}\left|
\begin{matrix}
    \left(1-1/\beta,1/\beta\right)\\ 
    (0,1)
  \end{matrix}
\right.\Bigg]\,\mathrm{d}x\\
&=&\frac{e^{\lambda^\beta t}}{\beta t^{1/\beta}}\frac{1}{2\pi i}\int_{c-i\infty}^{c+i\infty}\frac{\Gamma\left(1/\beta-w/\beta\right)}{\Gamma\left(1-w\right)t^{-w/\beta}}\int_0^\infty x^{s-w-1}e^{-\lambda x}\,\mathrm{d}x\,\mathrm{d}w\\
&=&\frac{e^{\lambda^\beta t}}{\beta t^{1/\beta}\lambda^s}\frac{1}{2\pi i}\int_{c-i\infty}^{c+i\infty}\frac{\Gamma\left(1/\beta-w/\beta\right)\Gamma\left(s-w\right)}{\Gamma\left(1-w\right)}\left(\frac{1}{t^{1/\beta}\lambda}\right)^{-w}\,\mathrm{d}w.
\end{eqnarray*}
The result now follows on using the definition of the $H$-function.
\end{proof}

\section{Independent inverse stable subordinators}
\setcounter{equation}{0}
The first-hitting time of a $\beta$-stable subordinator $E_\beta(t)$, also called inverse $\beta$-stable subordinator, is defined by
\begin{equation*}
E_\beta(t)=\inf\{x>0:D_\beta(x)>t\}.
\end{equation*}
The following relationship holds:
\begin{equation*}
\{E_\beta(t)\leq x\}=\{D_\beta(x)\geq t\}.
\end{equation*}
The integral representation for the density $g_\beta (x,t)$ of the inverse stable subordinator can be obtained on substituting $\lambda=0$ in Theorem 2.1 of Kumar and Vellaisamy \cite{Arun2015}.
\begin{proposition}\label{t11}
The density of the inverse $\beta$-stable subordinator, $0<\beta<1$, admits the following integral form
\begin{equation*}
g_\beta (x,t)=\frac{1}{\pi}\int_0^{\infty}u^{\beta-1}e^{-tu-xu^\beta\cos\beta\pi}\sin(\beta\pi-xu^\beta\sin\beta\pi)\,\mathrm{d}u,\ \ \ \ x>0.
\end{equation*}
\end{proposition}
\noindent Next result give the Mellin and Laplace transforms of the density of the inverse stable subordinator.
\begin{proposition}\label{t12}
The Mellin and Laplace transforms of the density of the inverse $\beta$-stable subordinator, $0<\beta<1/2$, are
\begin{equation*}
\mathcal{M}_x\left(g_\beta (x,t)\right)=\frac{1}{t^{(1-s)\beta}}\frac{\Gamma(s)}{\Gamma(1-\beta+\beta s)}
\end{equation*}
and
\begin{equation*}
\mathcal{L}_x\left(g_\beta (x,t)\right)=t^{\beta}\mathrm{H}^{1,1}_{1,2}\Bigg[t^\beta s\left|
\begin{matrix}
    (0,1)\\ 
    (0,1)&(0,\beta)
  \end{matrix}
\right.\Bigg],
\end{equation*}
respectively.
\end{proposition}
\begin{proof}
See Appendix \textbf{A.3}.
\end{proof}
On using the Mellin inversion formula and Proposition \ref{t12}, the density of the inverse stable subordinator can be represented in terms of the Fox's $H$-function as follows:
\begin{eqnarray}\label{ddd}
g_\beta (x,t)&=&\frac{1}{2\pi i}\int_{c-i\infty}^{c+i\infty}\mathcal{M}_x\left(g_\beta (x,t)\right)x^{-s}\,\mathrm{d}s\nonumber\\
&=&\frac{1}{t^{\beta}}\frac{1}{2\pi i}\int_{c-i\infty}^{c+i\infty}\frac{\Gamma(s)}{\Gamma(1-\beta+\beta s)}\left(\frac{x}{t^{\beta}}\right)^{-s}\,\mathrm{d}s\nonumber\\
&=&\frac{1}{t^{\beta}}\mathrm{H}^{1,0}_{1,1}\Bigg[\frac{x}{t^{\beta}}\left|
\begin{matrix}
    \left(1-\beta,\beta\right)\\ 
    (0,1)
  \end{matrix}
\right.\Bigg],\ \ \ \ x>0,\ \ 0<\beta<1/2.
\end{eqnarray}
The density of $E_\beta(t)$, $0<\beta<1$, can also be expressed in terms of the $M$-Wright function (see \cite[Eq. (5.7)]{Gorenflo2015}). Therefore,
\begin{eqnarray*}
g_\beta (x,t)&=&t^{-\beta}M_\beta(xt^{-\beta}),\ \ x>0\\
&=&t^{-\beta}{}_{0}\Psi_1\Bigg[-xt^{-\beta}\left|
\begin{matrix}
    \\ 
    (1-\beta,-\beta)
  \end{matrix}
\right.\Bigg]\\
&=&t^{-\beta}\frac{1}{2\pi i}\int_{c-i\infty}^{c+i\infty}\frac{\Gamma\left(s\right)}{\Gamma\left(1-\beta+\beta s\right)}\left(xt^{-\beta}\right)^{-s}\,\mathrm{d}s\\
&=&\frac{1}{t^{\beta}}\mathrm{H}^{1,0}_{1,1}\Bigg[\frac{x}{t^{\beta}}\left|
\begin{matrix}
    \left(1-\beta,\beta\right)\\ 
    (0,1)
  \end{matrix}
\right.\Bigg].
\end{eqnarray*}
Alternatively, the density of $E_\beta(t)$, $0<\beta< 1$, can be obtained as follows:
\begin{eqnarray*}
\mathbb{P}(E_\beta(t)\leq x)&=&1-\mathbb{P}(D_\beta(x)< t)\\
&=&1-\frac{1}{\beta}\mathrm{H}^{0,1}_{1,1}\Bigg[\frac{t}{x^{1/\beta}}\left|
\begin{matrix}
    (1,1/\beta)\\ 
    (0,1)
  \end{matrix}
\right.\Bigg],\ \ \ \ \mathrm{(using\ (\ref{cdf1}))}\\
&=&1-\mathrm{H}^{0,1}_{1,1}\Bigg[\frac{t^\beta}{x}\left|
\begin{matrix}
    (1,1)\\ 
    (0,\beta)
  \end{matrix}
\right.\Bigg],\ \ \ \ \mathrm{(using\ (\ref{2.9}))}\\
&=&1-\mathrm{H}^{1,0}_{1,1}\Bigg[\frac{x}{t^{\beta}}\left|
\begin{matrix}
    (1,\beta)\\ 
    (0,1)
  \end{matrix}
\right.\Bigg],\ \ \ \ \mathrm{(using\ (\ref{2.8}))}.
\end{eqnarray*}
Now differentiating with respect to $x$ both sides, we get
\begin{eqnarray*}
g_\beta (x,t)&=&\frac{1}{x}\mathrm{H}^{1,0}_{1,1}\Bigg[\frac{x}{t^{\beta}}\left|
\begin{matrix}
    (1,\beta)\\ 
    (1,1)
  \end{matrix}
\right.\Bigg],\ \ \ \ \mathrm{(using\ (\ref{2.11}))}\\
&=&\frac{1}{t^{\beta}}\mathrm{H}^{1,0}_{1,1}\Bigg[\frac{x}{t^{\beta}}\left|
\begin{matrix}
    \left(1-\beta,\beta\right)\\ 
    (0,1)
  \end{matrix}
\right.\Bigg],\ \ \ \ \mathrm{(using\ (\ref{2.10}))}.
\end{eqnarray*}
From above discussion it is clear that (\ref{ddd}) can be extended for index $\beta\in(0,1)$.
\begin{theorem}\label{t13}
The density of the inverse $\beta$-stable subordinator $E_\beta(t)$, $0<\beta<1$, in terms of the $H$-function is
\begin{equation}\label{4.1aa}
g_\beta (x,t)=\frac{1}{t^{\beta}}\mathrm{H}^{1,0}_{1,1}\Bigg[\frac{x}{t^{\beta}}\left|
\begin{matrix}
    \left(1-\beta,\beta\right)\\ 
    (0,1)
  \end{matrix}
\right.\Bigg],\ \ \ \ x>0.
\end{equation}
\end{theorem}
\begin{proposition}\label{t14}
The Mellin transform of the density of the inverse $\beta$-stable subordinator, $0<\beta<1$, with respect to the time variable is
\begin{equation*}
\mathcal{M}_t\left(g_\beta (x,t)\right)=\frac{1}{\beta x^{1-s/\beta}}\frac{\Gamma\left(1-s/\beta\right)}{\Gamma(1-s)}.
\end{equation*}
\end{proposition}
\begin{proof}
From (\ref{4.1aa}), we have
\begin{eqnarray}\label{4.2aa}
g_\beta (x,t)&=&\frac{1}{t^{\beta}}\mathrm{H}^{1,0}_{1,1}\Bigg[\frac{x}{t^{\beta}}\left|
\begin{matrix}
    \left(1-\beta,\beta\right)\\ 
    (0,1)
  \end{matrix}
\right.\Bigg]\nonumber\\
&=&\frac{1}{\beta t^{\beta}}\mathrm{H}^{1,0}_{1,1}\Bigg[\frac{x^{1/\beta}}{t}\left|
\begin{matrix}
    \left(1-\beta,1\right)\\ 
    (0,1/\beta)
  \end{matrix}
\right.\Bigg],\ \ \ \ \mathrm{(using\ (\ref{2.9}))}\nonumber\\
&=&\frac{1}{\beta x}\mathrm{H}^{1,0}_{1,1}\Bigg[\frac{x^{1/\beta}}{t}\left|
\begin{matrix}
    \left(1,1\right)\\ 
    (1,1/\beta)
  \end{matrix}
\right.\Bigg],\ \ \ \ \mathrm{(using\ (\ref{2.10}))}\nonumber\\
&=&\frac{1}{\beta x}\mathrm{H}^{0,1}_{1,1}\Bigg[\frac{t}{x^{1/\beta}}\left|
\begin{matrix}
    (0,1/\beta)\\ 
    (0,1)
  \end{matrix}
\right.\Bigg],
\end{eqnarray}
where the last step follows from (\ref{2.8}). The proof is complete on using (\ref{2.1p}).
\end{proof}
An alternate and direct proof of Proposition \ref{t14} is given in Appendix \textbf{A.4}. We note that on following this alternate proof for the Mellin transform and using the inverse Mellin formula, we can directly obtain the density of the inverse stable subordinator in the form (\ref{4.2aa}), which is equivalent to (\ref{4.1aa}).
\begin{remark}\label{t16}
On using (\ref{2.2p}), (\ref{2.7}) and (\ref{4.2aa}), we obtain another well known result
\begin{equation*}
\mathcal{L}_t\left(g_\beta (x,t)\right)=\frac{1}{\beta}x^{1/\beta-1}\mathrm{H}^{1,0}_{0,1}\Bigg[x^{1/\beta} s\left|
\begin{matrix}
    \\ 
    (1-1/\beta,1/\beta)
  \end{matrix}
\right.\Bigg]=s^{\beta-1}e^{-xs^\beta}.
\end{equation*}
\end{remark}
Next, the cumulative distribution function $G_\beta (x,t)$ of the inverse stable subordinator in terms of the Fox's $H$-function is obtained on using (\ref{2.11p}) and (\ref{4.1aa}).
\begin{proposition}
The cumulative distribution function of the inverse $\beta$-stable subordinator, $0<\beta<1$, is
\begin{equation*}
G_\beta (x,t)=\mathrm{H}^{1,1}_{2,2}\Bigg[\frac{x}{t^{\beta}}\left|
\begin{matrix}
    (1,1)&(1,\beta)\\ 
    (1,1)&(0,1)
  \end{matrix}
\right.\Bigg].
\end{equation*}
\end{proposition}
As obtained for independent stable subordinators, similar results for the product, quotient, power and scalar multiple of independent inverse stable subordinators holds.
\begin{theorem}
Let $E_{\beta_i}(t)$, $i=1,2,\ldots,n$, be $n$ independent inverse $\beta_i$-stable subordinators, $0<\beta_i<1$, with corresponding density functions
\begin{equation*}
g_{\beta_i}(x_i,t)=\frac{1}{t^{\beta_i}}\mathrm{H}^{1,0}_{1,1}\Bigg[\frac{x_i}{t^{\beta_i}}\left|
\begin{matrix}
    \left(1-\beta_i,\beta_i\right)\\ 
    (0,1)
  \end{matrix}
\right.\Bigg],\ \ \ \ x_i>0.
\end{equation*}
Then, the probability density function of\\
\noindent (i) $X(t)=aE_{\beta_i}(t),\ a>0,\ 1\leq i\leq n$, is given by
\begin{equation}\label{symm}
g(x,t)=\frac{1}{at^{\beta_i}}\mathrm{H}^{1,0}_{1,1}\Bigg[\frac{x}{at^{\beta_i}}\left|
\begin{matrix}
    \left(1-\beta_i,\beta_i\right)\\ 
    (0,1)
  \end{matrix}
\right.\Bigg],\ \ x>0.
\end{equation}
\noindent (ii) $Y(t)=\prod_{i=1}^nE_{\beta_i}(t)$ is given by
\begin{equation*}
g(y,t)=\frac{1}{t^{\sum_{i=1}^n\beta_i}}\mathrm{H}^{n,0}_{n,n}\Bigg[\frac{y}{t^{\sum_{i=1}^n\beta_i}}\left|
\begin{matrix}
    \left(1-\beta_i,\beta_i\right)_{1,n}\\ 
    (0,1)_{1,n}
  \end{matrix}
\right.\Bigg],\ \ y>0.
\end{equation*}
\noindent (iii) $W(t)=E_{\beta_i}^r(t),\ 1\leq i\leq n$, for $r$ rational is given by
\begin{equation*}
g(w,t)=\frac{1}{t^{r\beta_i}}\mathrm{H}^{1,0}_{1,1}\Bigg[\frac{w}{t^{r\beta_i}}\left|
\begin{matrix}
    \left(1-r\beta_i,r\beta_i\right)\\ 
    (1-r,r)
  \end{matrix}
\right.\Bigg],\ \  w>0,
\end{equation*}
when $r>0$ and for $r<0$
\begin{equation*}
g(w,t)=\frac{1}{t^{r\beta_i}}\mathrm{H}^{0,1}_{1,1}\Bigg[\frac{w}{t^{r\beta_i}}\left|
\begin{matrix}
    (r,-r)\\
    (r\beta_i,-r\beta_i)
      \end{matrix}
\right.\Bigg],\ \ w>0.
\end{equation*}
\noindent (iv) $Z(t)=E_{\beta_i}(t)/E_{\beta_j}(t),\ 1\leq i,j\leq n$ and $i\neq j$, is given by
\begin{equation*}
g(z,t)=\frac{1}{t^{\beta_i+\beta_j}}\mathrm{H}^{1,1}_{2,2}\Bigg[\frac{z}{t^{\beta_i+\beta_j}}\left|
\begin{matrix}
    \left(1-\beta_i,\beta_i\right)&(-1,1)\\ 
    (-\beta_j,\beta_j)&(0,1)
  \end{matrix}
\right.\Bigg],\ \ z>0.
\end{equation*}
\end{theorem}
\begin{proof}
(i) The density of $X$ is given by
\begin{eqnarray*}
g(x,t)&=&\frac{1}{2\pi i}\int_{c-i\infty}^{c+i\infty}\mathcal{M}_{x}(g(x,t))x^{-s}\,ds\\
&=&\frac{1}{2\pi i}\int_{c-i\infty}^{c+i\infty}a^{s-1}\mathcal{M}_{x_i}(g_{\beta_i}(x_i,t))x^{-s}\,ds\\
&=&\frac{1}{at^{\beta_i}}\frac{1}{2\pi i}\int_{c-i\infty}^{c+i\infty}\frac{\Gamma\left(s\right)}{\Gamma\left(1-\beta_i+\beta_is\right)}\left(\frac{x}{at^{\beta_i}}\right)^{-s}\,\mathrm{d}s,\ \ \ \ x>0,
\end{eqnarray*}
where the last step follows from Proposition \ref{t12}. The proof now follows on using (\ref{1.2}). Parts (ii), (iii) and (iv) follow from Theorems 4.1, 4.2 and 4.3 of Carter and Springer  \cite{Springer542}, respectively.
\end{proof}
\begin{remark}
Since $M$-Wright function is a particular case of $H$-function, (\ref{symm}) reduces to
\begin{equation}\label{symmz}
g(x,t)=\frac{1}{at^{\beta_i}}M_{\beta_i}\left(\frac{x}{at^{\beta_i}}\right),\ \ x>0.
\end{equation}
On substituting $a=1/t^{\beta_i}$ in (\ref{symm}) or (\ref{symmz}), we can obtain the self similarity property of the inverse stable subordinator \textit{i.e.} $E_{\beta_i}(t)/t^{\beta_i}{\overset{\mathcal{D}}{=}}E_{\beta_i}(1)$ for all $t>0$.
\end{remark}
\begin{remark}
\noindent From our results the Corollary 3.1 (a), (c) of \cite{Mark2004} \textit{i.e.} $E_\beta(t){\overset{\mathcal{D}}{=}}\left(D_\beta(1)/t\right)^{-\beta}$ and
\begin{equation*}
g_\beta (x,t)=\frac{t}{\beta}x^{-1-1/\beta}f_\beta \left(tx^{-1/\beta},1\right)
\end{equation*}
can easily be verified and an explicit form for $f_\beta \left(x,1\right)$ is given by (\ref{3.1aaa}).

Also, on using (\ref{2.1p}) and (\ref{4.1aa}), the $r$-th moment of the inverse stable subordinator is given by
\begin{equation*}
\mathbb{E}\left(E_\beta(t)^r\right)=C(\beta,r)t^{\beta r}=\frac{\Gamma(1+r)}{\Gamma(1+\beta r)}t^{\beta r},
\end{equation*}
which determines the positive finite constant $C(\beta,r)$ in Corollary 3.1 (b) of \cite{Mark2004}. Therefore, the mean and variance of $E_\beta(t)$ are explicitly given by
\begin{equation*}
\mathbb{E}\left(E_\beta(t)\right)=\frac{1}{\Gamma(1+\beta)}t^{\beta}
\end{equation*}
and
\begin{equation*}
\mathrm{Var}\left(E_\beta(t)\right)=\left(\frac{2}{\Gamma(1+2\beta)}-\frac{1}{\Gamma^2(1+\beta)}\right)t^{2\beta},
\end{equation*}
respectively.
\end{remark}
\vone\noindent{\bf Appendix}
\begin{proof}[A.1.]
By using the definition of the Mellin transform, we have 
\begin{eqnarray*}
\mathcal{M}_x\left(f_\beta (x,t)\right)&=&\int_{0}^{\infty}x^{s-1}f_\beta (x,t)\,\mathrm{d}x\\
&=&\frac{1}{\pi}\int_0^{\infty}\int_{0}^{\infty}x^{s-1}e^{-ux-tu^\beta\cos\beta\pi}\sin(tu^\beta\sin\beta\pi)\,\mathrm{d}u\,\mathrm{d}x\\
&=&\operatorname{Im}\frac{1}{\pi}\int_0^{\infty}\int_{0}^{\infty}x^{s-1}e^{-ux-tu^\beta\cos\beta\pi+itu^\beta\sin\beta\pi}\,\mathrm{d}u\,\mathrm{d}x\\
&=&\operatorname{Im}\frac{1}{\pi}\int_0^{\infty}e^{-tu^\beta (\cos\beta\pi-i\sin\beta\pi)}\int_{0}^{\infty}x^{s-1}e^{-ux}\,\mathrm{d}x\,\mathrm{d}u\\
&=&\operatorname{Im}\frac{\Gamma(s)}{\pi}\int_0^{\infty}u^{-s}e^{-tu^\beta e^{-i\beta\pi}}\,\mathrm{d}u\\
&=&\operatorname{Im}\frac{\Gamma(s)}{\beta\pi}\int_0^{\infty}w^{1/\beta-s/\beta-1}e^{-te^{-i\beta\pi}w}\,\mathrm{d}w,\ \ \ \ \mathrm{(}\mathrm{substituting}\ w=u^\beta\mathrm{)}\\
&=&\operatorname{Im}\frac{1}{\beta t^{(1-s)/\beta}}\frac{\Gamma\left(s\right)\Gamma\left((1-s)/\beta \right)e^{i(1-s)\pi}}{\pi}\\
&=&\frac{1}{\beta t^{(1-s)/\beta}}\frac{\Gamma\left(s\right)\Gamma\left((1-s)/\beta \right)\sin{(1-s)\pi}}{\pi}.
\end{eqnarray*}
The proof follows on using the identity (\ref{2.7sd}).
\end{proof}
\begin{proof}[A.2.]
An alternate and direct proof for the Mellin transform is
\begin{eqnarray*}
\mathcal{M}_t\left(f_\beta (x,t)\right)&=&\int_{0}^{\infty}t^{s-1}f_\beta (x,t)\,\mathrm{d}t\\
&=&\frac{1}{\pi}\int_0^{\infty}\int_{0}^{\infty}t^{s-1}e^{-ux-tu^\beta\cos\beta\pi}\sin(tu^\beta\sin\beta\pi)\,\mathrm{d}u\,\mathrm{d}t\\
&=&\operatorname{Im}\frac{1}{\pi}\int_0^{\infty}\int_{0}^{\infty}t^{s-1}e^{-ux-tu^\beta\cos\beta\pi+itu^\beta\sin\beta\pi}\,\mathrm{d}u\,\mathrm{d}t\\
&=&\operatorname{Im}\frac{1}{\pi}\int_0^{\infty}e^{-ux}\int_{0}^{\infty}t^{s-1}e^{-tu^\beta e^{-i\beta\pi}}\,\mathrm{d}t\,\mathrm{d}u\\
&=&\operatorname{Im}\frac{\Gamma(s)e^{i\beta s\pi}}{\pi}\int_0^{\infty}u^{1-\beta s-1}e^{-ux}\,\mathrm{d}u\\
&=&\frac{1}{x^{1-\beta s}}\frac{\Gamma\left(s\right)\Gamma\left(1-\beta s \right)\sin{\beta s\pi}}{\pi}.
\end{eqnarray*}
\end{proof}
\begin{proof}[A.3.]
By using the definition of the Mellin transform, we have
\begin{eqnarray*}
\mathcal{M}_x\left(g_\beta (x,t)\right)&=&\int_{0}^{\infty}x^{s-1}g_\beta (x,t)\,\mathrm{d}x\\
&=&\frac{1}{\pi}\int_0^{\infty}\int_{0}^{\infty}x^{s-1}u^{\beta-1}e^{-tu-xu^\beta\cos\beta\pi}\sin(\beta\pi-xu^\beta\sin\beta\pi)\,\mathrm{d}u\,\mathrm{d}x\\
&=&\operatorname{Im}\frac{1}{\pi}\int_0^{\infty}\int_{0}^{\infty}x^{s-1}u^{\beta-1}e^{-tu-xu^\beta\cos\beta\pi+i\beta\pi-ixu^\beta\sin\beta\pi}\,\mathrm{d}u\,\mathrm{d}x\\
&=&\operatorname{Im}\frac{e^{i\beta\pi}}{\pi}\int_0^{\infty}u^{\beta-1}e^{-tu}\int_{0}^{\infty}x^{s-1}e^{-u^\beta e^{i\beta\pi}x}\,\mathrm{d}x\,\mathrm{d}u\\
&=&\operatorname{Im}\frac{\Gamma(s)e^{i(1-s)\beta\pi}}{\pi}\int_0^{\infty}u^{(1-s)\beta-1}e^{-tu}\,\mathrm{d}u\\
&=&\operatorname{Im}\frac{1}{t^{(1-s)\beta}}\frac{\Gamma\left(s\right)\Gamma\left((1-s)\beta \right)e^{i(1-s)\beta\pi}}{\pi}\\
&=&\frac{1}{t^{(1-s)\beta}}\frac{\Gamma\left(s\right)\Gamma\left((1-s)\beta \right)\sin{(1-s)\beta\pi}}{\pi}\\
\end{eqnarray*}
The proof follows on using the identity (\ref{2.7sd}). For Laplace transform use (\ref{2.2p}).
\end{proof}
\begin{proof}[A.4.]
An alternate and direct proof is
\begin{eqnarray*}
\mathcal{M}_t\left(g_\beta (x,t)\right)&=&\int_{0}^{\infty}t^{s-1}g_\beta (x,t)\,\mathrm{d}t\\
&=&\frac{1}{\pi}\int_0^{\infty}\int_{0}^{\infty}t^{s-1}u^{\beta-1}e^{-tu-xu^\beta\cos\beta\pi}\sin(\beta\pi-xu^\beta\sin\beta\pi)\,\mathrm{d}u\,\mathrm{d}t\\
&=&\operatorname{Im}\frac{1}{\pi}\int_0^{\infty}u^{\beta-1}e^{-xu^\beta\cos\beta\pi+i\beta\pi-ixu^\beta\sin\beta\pi}\int_{0}^{\infty}t^{s-1}e^{-ut}\,\mathrm{d}t\,\mathrm{d}u\\
&=&\operatorname{Im}\frac{\Gamma(s)e^{i\beta\pi}}{\pi}\int_0^{\infty}u^{\beta-s-1}e^{-xe^{i\beta\pi}u^\beta}\,\mathrm{d}u\\
&=&\operatorname{Im}\frac{\Gamma(s)e^{i\beta\pi}}{\beta\pi}\int_0^{\infty}w^{-s/\beta}e^{-xe^{i\beta\pi}w}\,\mathrm{d}w,\ \ \ \ \mathrm{(substituting}\ w=u^\beta\mathrm{)}\\
&=&\frac{1}{\beta x^{1-s/\beta}}\frac{\Gamma\left(s\right)\Gamma\left(1-s/\beta \right)\sin{s\pi}}{\pi}.
\end{eqnarray*}
\end{proof}
\printbibliography
\end{document}